\documentclass[reqno,12pt]{amsart}
\usepackage{amscd,amsfonts,amssymb}
\numberwithin{equation}{section}
\newtheorem{Theorem}{Theorem}[section]
\newtheorem{Lemma}{Lemma}[section]

\theoremstyle{definition}
\newtheorem{Definition}{Definition}[section]
\theoremstyle{remark}
\newtheorem{Remark}{Remark}[section]
\newtheorem{Example}{Example}[section]

\newcommand{\essinf}{\mathop{\rm ess \, inf}\limits}
\newcommand{\esssup}{\mathop{\rm ess \, sup}\limits}
\newcommand{\supp}{\mathop{\rm supp}\nolimits}

\newcommand{\mes}{\mathop{\rm mes}\nolimits}

\author{A.A. Kon'kov}
\address{Department of Differential Equations,
Faculty of Mechanics and Mathematics,
Mo\-s\-cow Lo\-mo\-no\-sov State University,
Vorobyovy Gory,
Moscow, 119992 Russia}
\email{konkov@mech.math.msu.su}
\author{A.E. Shishkov}
\address{
Center of Nonlinear Problems of Mathematical Physics,
RUDN University,
Miklukho-Maklaya str. 6,
Moscow, 117198 Russia;
Institute of Applied Mathematics and Mechanics of NAS of Ukraine,
Dobrovol'skogo str. 1, Slavyansk, 84116 Ukraine
}
\email{aeshkv@yahoo.com}
\thanks{The work of the second author is supported by RUDN University, Project 5-100}
\title[On stabilization of solutions]{On stabilization of solutions of higher order evolution inequalities}
\thanks{}
\date{}

\begin{document}

\begin{abstract}
We obtain sharp conditions guaranteeing that every non-negative weak solution of the inequality 
$$
	\sum_{|\alpha| = m}
	\partial^\alpha
	a_\alpha (x, t, u)
	-
	u_t
	\ge
	f (x, t) g (u)
	\quad
	\mbox{in } {\mathbb R}_+^{n+1} = {\mathbb R}^n \times (0, \infty),
	\quad
	m,n \ge 1,
$$
stabilizes to zero as $t \to \infty$. These conditions generalize the well-known Keller-Osserman
condition on the grows of the function $g$ at infinity.
\end{abstract}

\maketitle

\section{Introduction}
We study non-negative solutions of the inequality
\begin{equation}
	\sum_{|\alpha| = m}
	\partial^\alpha
	a_\alpha (x, t, u)
	-
	u_t
	\ge
	f (x, t) g (u)
	\quad
	\mbox{in } {\mathbb R}_+^{n+1} = {\mathbb R}^n \times (0, \infty),
	\quad
	m,n \ge 1,
	\label{1.1}
\end{equation}
where $a_\alpha$ are Caratheodory functions such that
\begin{equation}
	|a_\alpha (x, t, \zeta)| 
	\le 
	A \zeta^p,
	\quad
	|\alpha| = m,
	\label{1.2}
\end{equation}
with some constants $A > 0$ and $p \ge 1$ for almost all $(x, t) \in {\mathbb R}_+^{n+1}$ and for all $\zeta \in [0, \infty)$.
In addition, it is assumed that $f$ is a measurable function on the set ${\mathbb R}_+^{n+1}$,
$g (\zeta^{1 / p})$ is a non-decreasing convex function on the closed interval $[0, \infty)$, and $g (\zeta) > 0$ for all $\zeta > 0$.
As is customary, by $\alpha$ we mean a multi-index $\alpha = {(\alpha_1, \ldots, \alpha_n)}$ with
$|\alpha| = \alpha_1 + \ldots + \alpha_n$
and
$\partial^\alpha = {\partial^{|\alpha|} / \partial_{x_1}^{\alpha_1} \ldots \partial_{x_n}^{\alpha_n}}$,
$x = {(x_1, \ldots, x_n)}$.

\begin{Definition}\label{D1.1}
A non-negative function $u \in L_{p, loc} ({\mathbb R}_+^{n+1})$ is called a weak solution of~\eqref{1.1} if 
$f (x, t) g (u) \in L_{1, loc} ({\mathbb R}_+^{n+1})$ and, moreover,
for any non-negative function $\varphi \in C_0^\infty ({\mathbb R}_+^{n+1})$
the following inequality holds:
\begin{align}
	&
	\int_{
		{\mathbb R}_+^{n+1}
	}
	\sum_{|\alpha| = m}
	(- 1)^m
	a_\alpha (x, t, u)
	\partial^\alpha
	\varphi
	\,
	dx
	dt
	+
	\int_{
		{\mathbb R}_+^{n+1}
	}
	u
	\varphi_t
	\,
	dx
	dt
	\nonumber
	\\
	&
	\qquad
	\ge
	\int_{
		{\mathbb R}_+^{n+1}
	}
	f (x, t) g (u)
	\varphi
	\,
	dx
	dt.
	\label{1.3}
\end{align}
\end{Definition}

Questions treated in this paper were earlier investigated
mainly for differential operators of the second order~[1--14].
The case of higher order operators has been studied much less~\cite{B, GS}.
Our aim is to obtain sufficient stabilization conditions for weak solution 
of inequality~\eqref{1.1}.
In so doing, no initial conditions on solutions of~\eqref{1.1} are imposed. 
We even admit that $u (x, t)$ can tend to infinity as $t \to +0$.
We also impose no ellipticity conditions on the coefficients $a_\alpha$ of the differential operator. Thus, our results can be applied to both parabolic and so-called anti-parabolic inequalities.

\section{Main results}

\begin{Theorem}\label{T2.1}
Let
\begin{equation}
	\int_1^\infty
	g^{- 1 / m} (\zeta)
	\zeta^{p / m - 1}
	\,
	d\zeta
	<
	\infty
	\label{T2.1.1}
\end{equation}
and
\begin{equation}
	\lim_{t \to \infty}
	\essinf_{
		K \times (t, \infty)
	}
	f
	=
	\infty
	\label{T2.1.2}
\end{equation}
for any compact set $K \subset {\mathbb R}^n$.
Then every non-negative weak solution of~\eqref{1.1} stabilizes to zero as $t \to \infty$ 
in the $L_1$ norm on an arbitrary compact set $K \subset {\mathbb R}^n$, i.e.
\begin{equation}
	\lim_{t \to \infty}
	\esssup_{\tau \in (t, \infty)}
	\int_K
	u (x, \tau)
	\,
	dx
	=
	0.
	\label{T2.1.3}
\end{equation}
\end{Theorem}

The proof of Theorem~\ref{T2.1} is given in Section~\ref{proofOfTheorem}.

\begin{Remark}\label{R2.1}
Since $u \in L_{1, loc} ({\mathbb R}_+^{n+1})$, the integral on the left in~\eqref{T2.1.3} 
is defined for almost all $\tau \in (0, \infty)$.
\end{Remark}

\begin{Example}\label{E2.1}
Consider the inequality
\begin{equation}
	\Delta^{m / 2} u - u_t \ge f (x, t) u^\lambda
	\quad
	\mbox{in } {\mathbb R}_+^{n+1},
	\label{E2.1.1}
\end{equation}
where $m$ is a positive even integer and $\lambda$ is a real number.
By Theorem~\ref{T2.1}, if 
\begin{equation}
	\lambda > 1
	\label{E2.1.2}
\end{equation}
and~\eqref{T2.1.2} is valid, 
then every non-negative weak solution of~\eqref{E2.1.1} stabilizes to zero as $t \to \infty$ 
in $L_1$ norm on an arbitrary compact subset of ${\mathbb R}^n$.

Now, let us consider the inequality
\begin{equation}
	\Delta^{m / 2} u - u_t \ge f (x, t) u \ln^\nu (1 + u)
	\quad
	\mbox{in } {\mathbb R}_+^{n+1},
	\label{E2.1.3}
\end{equation}
where $\nu$ is a real number. In other words, we examine the case of the critical exponent 
$\lambda = 1$ in the right-hand side of~\eqref{E2.1.1} ``spoiled'' by the logarithm.
As before, we assume that $m$ is a positive even integer.

It can easily be seen that~\eqref{T2.1.1} is equivalent to the condition 
\begin{equation}
	\nu > m.
	\label{E2.1.4}
\end{equation}
Thus, if~\eqref{T2.1.2} and~\eqref{E2.1.4} are valid, 
then Theorem~\ref{T2.1} implies that every non-negative weak solution of~\eqref{E2.1.3}
stabilizes to zero as $t \to \infty$ in $L_1$ norm on an arbitrary compact subset of ${\mathbb R}^n$.

Condition~\eqref{E2.1.4} is the best possible. 
Really, let us show that there exists a positive function
$f \in {C ({\mathbb R}^n \times [0, \infty))}$ for which~\eqref{T2.1.2} holds and the inequality
\begin{equation}
	\Delta^{m / 2} u - u_t \ge f (x, t) u \ln^m (1 + u)
	\quad
	\mbox{in } {\mathbb R}_+^{n+1}
	\label{E2.1.5}
\end{equation}
has a classical solution satisfying the bound 
$u (x, t) \ge e$ for all $(x, t) \in {\mathbb R}_+^{n+1}$.
It is obvious that this solution is also a solution of~\eqref{E2.1.3} for all $\nu \le m$.
We shall seek it in the form 
$$
	u (x, t)
	=
	e^{
		e^{
			w (x, t)
		}
	},
$$
where
$$
	w (x, t)
	=
	(t + 1)
	\sum_{i=1}^n
	e^{x_i}.
$$
By direct differentiation, one can verify that
$$
	\Delta^{m / 2} u
	\ge
	\sum_{i=1}^n
	(t + 1)^m
	e^{m x_i}
	e^{m w (x, t)}
	u
	+
	\sum_{i=1}^n
	(t + 1)
	e^{x_i}
	e^{w (x, t)}
	u
$$
for all $(x, t) \in {\mathbb R}_+^{n+1}$.
Since
$$
	u_t
	=
	\sum_{i=1}^n
	e^{x_i}
	e^{w (x, t)}
	u,
$$
this yields
\begin{equation}
	\Delta^{m / 2} u - u_t
	\ge
	\sum_{i=1}^n
	(t + 1)^m
	e^{m x_i}
	e^{m w (x, t)}
	u
	\label{E2.1.6}
\end{equation}
for all $(x, t) \in {\mathbb R}_+^{n+1}$.
It can be seen that
$$
	e^{w (x, t)} 
	= 
	\ln u 
	\ge
	\frac{1}{2}
	\ln (1 + u)
$$
for all $(x, t) \in {\mathbb R}_+^{n+1}$.
Thus,~\eqref{E2.1.6} implies inequality~\eqref{E2.1.5} with
$$
	f (x, t)
	=
	\frac{1}{2^m}
	\sum_{i=1}^n
	(t + 1)^m
	e^{m x_i}.
$$

Note that, along with~\eqref{E2.1.4}, we have established the exactness of condition~\eqref{E2.1.2}.
In fact, any solution of~\eqref{E2.1.5} satisfying the inequality $u \ge e$ on the set 
${\mathbb R}_+^{n+1}$ is also a solution of~\eqref{E2.1.1} for all $\lambda \le 1$.
\end{Example}

\begin{Remark}\label{R2.2}
If $m = 2$ and $p = 1$, then~\eqref{T2.1.1} takes the form
\begin{equation}
	\int_1^\infty
	(g (t) t)^{- 1 / 2}
	\,
	dt
	<
	\infty.
	\label{R2.2.1}
\end{equation}
It is easy to see that~\eqref{R2.2.1} is equivalent to the well-known Keller-Osserman condition 
\begin{equation}
	\int_1^\infty
	\left(
		\int_1^t
		g (s)
		\,
		ds
	\right)^{-1/2}
	\,
	dt
	<
	\infty
	\label{R2.2.2}
\end{equation}
on the grows of the function $g$ at infinity~\cite{Keller, Osserman} 
which plays an important role in the theory of semilinear elliptic and parabolic equations (see, for instance,~\cite{SV} and references therein). 
Really, since $g$ is a non-decreasing positive function on the interval $(0, \infty)$, we have
$$
	\int_1^t
	g (s)
	\,
	ds
	\ge
	\int_{t/2}^t
	g (s)
	\,
	ds
	\ge
	\frac{t}{2} 
	g
	\left(
		\frac{t}{2}
	\right),
	\quad
	t > 2.
$$
Hence,~\eqref{R2.2.1} implies~\eqref{R2.2.2}.
On the other hand,
$$
	\int_1^t
	g (s)
	\,
	ds
	\le
	t g (t),
	\quad
	t > 1;
$$
therefore,~\eqref{R2.2.1} follows from~\eqref{R2.2.2}.

We can in this context call~\eqref{T2.1} as a generalized Keller-Osserman condition.
In Example~\ref{E2.1}, it is shown that this condition is the best possible.
We put forward a hypothesis that~\eqref{T2.1} is also a necessary stabilization condition for solutions of inequality~\eqref{1.1}.
\end{Remark}

\section{Proof of Theorem~\ref{T2.1}}\label{proofOfTheorem}

Below, it is assumed that $u$ is a non-negative weak solution of~\eqref{1.1}.
By $C$ we mean various positive constants that can depend only on $m$, $n$, and $p$.

Let us use the following notations. We denote 
$B_r = \{ x \in {\mathbb R}^n : |x| < r \}$ 
and
$Q_r^{t_1, t_2} = \{ (x, t) \in {\mathbb R}^{n + 1}: |x| < r, t_1 < t < t_2 \}$.
Further, let $\omega \in C^\infty ({\mathbb R})$ be a non-negative function such that
$\supp \omega \subset (-1, 1)$ and
$$
	\int_{-\infty}^\infty
	\omega
	\,
	d t
	=
	1.
$$
We need the Steklov-Schwartz averaging kernel
\begin{equation}
	\omega_h (t) 
	= 
	\frac{1}{h} 
	w
	\left(
		\frac{t}{h}
	\right),
	\quad
	h > 0.
	\label{3.1}
\end{equation}

\begin{Lemma}\label{L3.1}
Let $0 < r_1 < r_2$ and $0 < h < \tau_1 < \tau_2 < \tau$ be some real numbers with $r_2 \le 2 r_1$. 
If $f (x, t) \ge 0$ for almost all $(x,t) \in Q_{r_2}^{\tau - \tau_2, \tau + h}$, then
\begin{align}
	&
	\frac{1}{(r_2 - r_1)^m}
	\int_{
		Q_{r_2}^{\tau - \tau_2, \tau + h}
		\setminus
		Q_{r_1}^{\tau - \tau_2, \tau + h}
	}
	u^p
	\,
	dx
	dt
	+
	\frac{1}{\tau_2 - \tau_1}
	\int_{
		Q_{r_2}^{\tau - \tau_2, \tau - \tau_1}
	}
	u
	\,
	dx
	dt
	\nonumber
	\\
	&
	\qquad
	\ge
	C
	\left(
		\int_{
			Q_{r_1}^{\tau - \tau_1, \tau - h}
		}
		f (x, t)
		g (u)
		\,
		dx
		dt
		+
		\int_{
			Q_{r_1}^{\tau - h, \tau + h}
		}
		\omega_h (\tau - t)
		u
		\,
		dx
		dt
	\right).
	\label{L3.1.1}
\end{align}
\end{Lemma}

\begin{proof}
We take a non-decreasing function $\varphi_0 \in C^\infty ({\mathbb R})$ such that
$$
	\left.
		\varphi_0
	\right|_{
		(- \infty, 0]
	}
	=
	0
	\quad
	\mbox{and}
	\quad
	\left.
		\varphi_0
	\right|_{
		[1, \infty)
	}
	=
	1.
$$
Also let
$$
	\eta (t)
	=
	\int_t^\infty
	\omega_h (\tau - \xi)
	\,
	d\xi.
$$
Using
$$
	\varphi (x, t)
	=
	\varphi_0
	\left(
		\frac{r_2 - |x|}{r_2 - r_1}
	\right)
	\varphi_0
	\left(
		\frac{t - \tau + \tau_2}{\tau_2 - \tau_1}
	\right)
	\eta (t)
$$
as a test function in~\eqref{1.3}, we obtain
\begin{align}
	&
	\int_{
		{\mathbb R}_+^{n+1}
	}
	\sum_{|\alpha| = m}
	(- 1)^m
	a_\alpha (x, t, u)
	\partial^\alpha
	\varphi_0
	\left(
		\frac{r_2 - |x|}{r_2 - r_1}
	\right)
	\varphi_0
	\left(
		\frac{t - \tau + \tau_2}{\tau_2 - \tau_1}
	\right)
	\eta (t)
	\,
	dx
	dt
	\nonumber
	\\
	&
	+
	\int_{
		{\mathbb R}_+^{n+1}
	}
	u
	\varphi_0
	\left(
		\frac{r_2 - |x|}{r_2 - r_1}
	\right)
	\frac{\partial \varphi_0}{\partial t}
	\left(
		\frac{t - \tau + \tau_2}{\tau_2 - \tau_1}
	\right)
	\eta (t)
	\,
	dx
	dt
	\nonumber
	\\
	&
	\qquad
	\ge
	\int_{
		{\mathbb R}_+^{n+1}
	}
	f (x, t) g (u)
	\varphi_0
	\left(
		\frac{r_2 - |x|}{r_2 - r_1}
	\right)
	\varphi_0
	\left(
		\frac{t - \tau + \tau_2}{\tau_2 - \tau_1}
	\right)
	\eta (t)
	\,
	dx
	dt
	\nonumber
	\\
	&
	\qquad
	\phantom{{}\ge{}}
	+
	\int_{
		{\mathbb R}_+^{n+1}
	}
	u
	\varphi_0
	\left(
		\frac{r_2 - |x|}{r_2 - r_1}
	\right)
	\varphi_0
	\left(
		\frac{t - \tau + \tau_2}{\tau_2 - \tau_1}
	\right)
	\omega_h (\tau - t)
	\,
	dx
	dt.
	\label{PL3.1.1}
\end{align}
Condition~\eqref{1.2} allows us to assert that
\begin{align*}
	&
	\frac{1}{(r_2 - r_1)^m}
	\int_{
		Q_{r_2}^{\tau - \tau_2, \tau + h}
		\setminus
		Q_{r_1}^{\tau - \tau_2, \tau + h}
	}
	u^p
	\,
	dx
	dt
	\\
	&
	\qquad
	\ge
	C
	\int_{
		{\mathbb R}_+^{n+1}
	}
	\sum_{|\alpha| = m}
	\left|
		a_\alpha (x, t, u)
		\partial^\alpha
		\varphi_0
		\left(
			\frac{r_2 - |x|}{r_2 - r_1}
		\right)
		\varphi_0
		\left(
			\frac{t - \tau + \tau_2}{\tau_2 - \tau_1}
		\right)
		\eta (t)
	\right|
	\,
	dx
	dt.
\end{align*}
In so doing, we obviously have
$$
	\frac{1}{\tau_2 - \tau_1}
	\int_{
		Q_{r_2}^{\tau - \tau_2, \tau - \tau_1}
	}
	u
	\,
	dx
	dt
	\ge
	C
	\int_{
		{\mathbb R}_+^{n+1}
	}
	u
	\varphi_0
	\left(
		\frac{r_2 - |x|}{r_2 - r_1}
	\right)
	\frac{\partial \varphi_0}{\partial t}
	\left(
		\frac{t - \tau + \tau_2}{\tau_2 - \tau_1}
	\right)
	\eta (t)
	\,
	dx
	dt
$$
and
\begin{align*}
	&
	\int_{
		{\mathbb R}_+^{n+1}
	}
	f (x, t) g (u)
	\varphi_0
	\left(
		\frac{r_2 - |x|}{r_2 - r_1}
	\right)
	\varphi_0
	\left(
		\frac{t - \tau + \tau_2}{\tau_2 - \tau_1}
	\right)
	\eta (t)
	\,
	dx
	dt
	\\
	&
	\qquad
	\ge
	\int_{
		Q_{r_1}^{\tau - \tau_1, \tau - h}
	}
	f (x, t)
	g (u)
	\,
	dx
	dt.
\end{align*}
Finally, since
$
	\omega_h (\tau - t) = 0
$ 
for all $t \not\in (\tau - h, \tau + h)$ and
${\varphi_0	((t - \tau + \tau_2) / (\tau_2 - \tau_1))} = 1$
for all $t > \tau - h$, the second summand in the right-hand side of~\eqref{PL3.1.1}
can be estimated as follows:
\begin{align*}
	&
	\int_{
		{\mathbb R}_+^{n+1}
	}
	u
	\varphi_0
	\left(
		\frac{r_2 - |x|}{r_2 - r_1}
	\right)
	\varphi_0
	\left(
		\frac{t - \tau + \tau_2}{\tau_2 - \tau_1}
	\right)
	\omega_h (\tau - t)
	\,
	dx
	dt
	\\
	&
	\qquad
	\ge
	\int_{
		Q_{r_1}^{\tau - h, \tau + h}
	}
	\omega_h (\tau - t)
	u 
	\,
	dx
	dt.
\end{align*}
Thus, combining~\eqref{PL3.1.1} with the last four inequalities, we deduce~\eqref{L3.1.1}.
\end{proof}

For any real number $R > 0$ we define the function
$$
	J_R (r, \tau)
	=
	\frac{
		1
	}{
		\mes Q_{2 R}^{\tau - 2^m R^m, \tau}
	}
	\int_{
		Q_r^{\tau - r^m, \tau}
	}
	u^p
	\,
	dx
	dt,
$$
where $0 < r \le 2 R$ and $\tau > 2^m R^m$.
Also put
$$
	G (\zeta) = g (\zeta^{1 / p}),
	\quad
	\zeta \ge 0.
$$

\begin{Lemma}\label{L3.2}
Let $R > 0$ and $\tau > 0$ be some real numbers such that $\tau > 2^m R^m$ and
$f (x, t) \ge 0$ for almost all $(x, t) \in Q_{2 R}^{\tau - 2^m R^m, \infty}$.
Then for all $r_1$ and $r_2$ satisfying the condition $R \le r_1 < r_2 \le 2 R$
at least one of the following two inequalities is valid:
\begin{equation}
	J_R (r_2, \tau) - J_R (r_1, \tau)
	\ge
	C 
	(r_2 - r_1)^m
	G (J_R (r_1, \tau))
	\essinf_{
		Q_{2 R}^{\tau - 2^m R^m, \tau}
	}
	f,
	\label{L3.2.1}
\end{equation}
\begin{equation}
	J_R (r_2, \tau) - J_R (r_1, \tau)
	\ge
	C 
	R^{m p - 1}
	(r_2 - r_1)
	G^p (J_R (r_1, \tau))
	\essinf_{
		Q_{2 R}^{\tau - 2^m R^m, \tau}
	}
	f^p.
	\label{L3.2.2}
\end{equation}
\end{Lemma}

\begin{proof}
Inequality~\eqref{L3.1.1} of Lemma~\ref{L3.1} with $\tau_1 = r_1^m$ and $\tau_2 = r_2^m$ yields
\begin{align*}
	&
	\frac{1}{(r_2 - r_1)^m}
	\int_{
		Q_{r_2}^{\tau - r_2^m, \tau + h}
		\setminus
		Q_{r_1}^{\tau - r_2^m, \tau + h}
	}
	u^p
	\,
	dx
	dt
	+
	\frac{1}{r_2^m - r_1^m}
	\int_{
		Q_{r_2}^{\tau - r_2^m, \tau - r_1^m}
	}
	u
	\,
	dx
	dt
	\\
	&
	\qquad
	\ge
	C
	\int_{
		Q_{r_1}^{\tau - r_1^m, \tau - h}
	}
	f (x, t)
	g (u)
	\,
	dx
	dt
\end{align*}
for all sufficiently small $h > 0$.
Note that the second summand in the right-hand side of~\eqref{L3.1.1} is non-negative; 
therefore, it can be dropped. 
Passing to the limit as $h \to +0$ in the last estimate, we obviously obtain
\begin{align}
	&
	\frac{1}{(r_2 - r_1)^m}
	\int_{
		Q_{r_2}^{\tau - r_2^m, \tau}
		\setminus
		Q_{r_1}^{\tau - r_2^m, \tau}
	}
	u^p
	\,
	dx
	dt
	+
	\frac{1}{r_2^m - r_1^m}
	\int_{
		Q_{r_2}^{\tau - r_2^m, \tau - r_1^m}
	}
	u
	\,
	dx
	dt
	\nonumber
	\\
	&
	\qquad
	\ge
	C
	\int_{
		Q_{r_1}^{\tau - r_1^m, \tau}
	}
	f (x, t)
	g (u)
	\,
	dx
	dt.
	\label{PL3.2.1}
\end{align}

Assume that
\begin{equation}
	\frac{1}{(r_2 - r_1)^m}
	\int_{
		Q_{r_2}^{\tau - r_2^m, \tau}
		\setminus
		Q_{r_1}^{\tau - r_2^m, \tau}
	}
	u^p
	\,
	dx
	dt
	\ge
	\frac{1}{r_2^m - r_1^m}
	\int_{
		Q_{r_2}^{\tau - r_2^m, \tau - r_1^m}
	}
	u
	\,
	dx
	dt.
	\label{PL3.2.a2}
\end{equation}
Then~\eqref{PL3.2.1} implies the inequality
\begin{equation}
	\int_{
		Q_{r_2}^{\tau - r_2^m, \tau}
		\setminus
		Q_{r_1}^{\tau - r_2^m, \tau}
	}
	u^p
	\,
	dx
	dt
	\ge
	C
	(r_2 - r_1)^m
	\int_{
		Q_{r_1}^{\tau - r_1^m, \tau}
	}
	f (x, t)
	g (u)
	\,
	dx
	dt.
	\label{PL3.2.a5}
\end{equation}
We have
$$
	Q_{r_2}^{\tau - r_2^m, \tau}
	\setminus
	Q_{r_1}^{\tau - r_2^m, \tau}
	\subset
	Q_{r_2}^{\tau - r_2^m, \tau}
	\setminus
	Q_{r_1}^{\tau - r_1^m, \tau};
$$
therefore,
\begin{align*}
	J_R (r_2, \tau) - J_R (r_1, \tau)
	&
	=
	\frac{
		1
	}{
		\mes Q_{2 R}^{\tau - 2^m R^m, \tau}
	}
	\int_{
		Q_{r_2}^{\tau - r_2^m, \tau}
		\setminus
		Q_{r_1}^{\tau - r_1^m, \tau}
	}
	u^p
	\,
	dx
	dt
	\\
	&
	\ge
	\frac{
		1
	}{
		\mes Q_{2 R}^{\tau - 2^m R^m, \tau}
	}
	\int_{
		Q_{r_2}^{\tau - r_2^m, \tau}
		\setminus
		Q_{r_1}^{\tau - r_2^m, \tau}
	}
	u^p
	\,
	dx
	dt.
\end{align*}
Combining this with~\eqref{PL3.2.a5}, we obtain
\begin{equation}
	J_R (r_2, \tau) - J_R (r_1, \tau)
	\ge
	\frac{
		C (r_2 - r_1)^m
	}{
		\mes Q_{2 R}^{\tau - 2^m R^m, \tau}
	}
	\int_{
		Q_{r_1}^{\tau - r_1^m, \tau}
	}
	f (x, t)
	g (u)
	\,
	dx
	dt.
	\label{PL3.2.2}
\end{equation}
Let us establish the validity of the estimate
\begin{equation}
	\int_{
		Q_{r_1}^{\tau - r_1^m, \tau}
	}
	f (x, t)
	g (u)
	\,
	dx
	dt
	\ge
	G (J_R (r_1, \tau))
	\mes Q_{r_1}^{\tau - r_1^m, \tau}
	\essinf_{
		Q_{r_1}^{\tau - r_1^m, \tau}
	}
	f.
	\label{PL3.2.3}
\end{equation}
Really, if
$$
	\essinf_{
		Q_{r_1}^{\tau - r_1^m, \tau}
	}
	f
	=
	0,
$$
then~\eqref{PL3.2.3} is evident; therefore, it can be assumed without loss of generality that
$$
	\essinf_{
		Q_{r_1}^{\tau - r_1^m, \tau}
	}
	f
	>
	0.
$$
In this case, we have
$$
	\int_{
		Q_{r_1}^{\tau - r_1^m, \tau}
	}
	g (u)
	\,
	dx
	dt
	<
	\infty.
$$
since
$$
	\int_{
		Q_{r_1}^{\tau - r_1^m, \tau}
	}
	f (x, t)
	g (u)
	\,
	dx
	dt
	<
	\infty.
$$
In so doing, we obviously obtain
\begin{equation}
	\int_{
		Q_{r_1}^{\tau - r_1^m, \tau}
	}
	f (x, t)
	g (u)
	\,
	dx
	dt
	\ge
	\essinf_{
		Q_{r_1}^{\tau - r_1^m, \tau}
	}
	f
	\int_{
		Q_{r_1}^{\tau - r_1^m, \tau}
	}
	g (u)
	\,
	dx
	dt.
	\label{PL3.2.a1}
\end{equation}
Since $G$ is a non-decreasing convex function, one can assert that
\begin{align*}
	&
	\frac{
		1
	}{
		\mes Q_{r_1}^{\tau - r_1^m, \tau}
	}
	\int_{
		Q_{r_1}^{\tau - r_1^m, \tau}
	}
	g (u)
	\,
	dx
	dt
	=
	\frac{
		1
	}{
		\mes Q_{r_1}^{\tau - r_1^m, \tau}
	}
	\int_{
		Q_{r_1}^{\tau - r_1^m, \tau}
	}
	G (u^p)
	\,
	dx
	dt
	\\
	&
	\qquad
	\ge
	G
	\left(
		\frac{
			1
		}{
			\mes Q_{r_1}^{\tau - r_1^m, \tau}
		}
		\int_{
			Q_{r_1}^{\tau - r_1^m, \tau}
		}
		u^p
		\,
		dx
		dt
	\right)
	\ge
	G (J_R (r_1, \tau)),
\end{align*}
whence in turn it follows that
$$
	\int_{
		Q_{r_1}^{\tau - r_1^m, \tau}
	}
	g (u)
	\,
	dx
	dt
	\ge
	G (J_R (r_1, \tau))
	\mes Q_{r_1}^{\tau - r_1^m, \tau}.
$$
In view of~\eqref{PL3.2.a1}, this implies~\eqref{PL3.2.3}.

Further, it does not present any particular problem to verify that
\begin{equation}
	A_1 R^{m + n} \le \mes Q_r^{\tau - r^m, \tau} \le A_2 R^{m + n}
	\label{PL3.2.a6}
\end{equation}
for all $R \le r \le 2 R$, where the constants $A_1 > 0$ and $A_2 > 0$ depend only on $m$ and $n$.
Thus, combining~\eqref{PL3.2.2} and~\eqref{PL3.2.3}, we arrive at~\eqref{L3.2.1}.

Now, assume that the opposite inequality to~\eqref{PL3.2.a2} holds, i.e.
$$
	\frac{1}{(r_2 - r_1)^m}
	\int_{
		Q_{r_2}^{\tau - r_2^m, \tau}
		\setminus
		Q_{r_1}^{\tau - r_2^m, \tau}
	}
	u^p
	\,
	dx
	dt
	<
	\frac{1}{r_2^m - r_1^m}
	\int_{
		Q_{r_2}^{\tau - r_2^m, \tau - r_1^m}
	}
	u
	\,
	dx
	dt.
$$
Combining this with~\eqref{PL3.2.1}, we have
\begin{equation}
	\int_{
		Q_{r_2}^{\tau - r_2^m, \tau - r_1^m}
	}
	u
	\,
	dx
	dt
	\ge
	C
	(r_2^m - r_1^m)
	\int_{
		Q_{r_1}^{\tau - r_1^m, \tau}
	}
	f (x, t)
	g (u)
	\,
	dx
	dt.
	\label{PL3.2.4}
\end{equation}
By the H\"older inequality, one can show that
$$
	\int_{
		Q_{r_2}^{\tau - r_2^m, \tau - r_1^m}
	}
	u
	\,
	dx
	dt
	\le
	\left(
		\mes Q_{r_2}^{\tau - r_2^m, \tau - r_1^m}
	\right)^{(p - 1) / p}
	\left(
		\int_{
			Q_{r_2}^{\tau - r_2^m, \tau - r_1^m}
		}
		u^p
		\,
		dx
		dt
	\right)^{1 / p},
$$
whence, taking into account the fact that
$$
	\mes Q_{r_2}^{\tau - r_2^m, \tau - r_1^m}
	\le
	C
	(r_2^m - r_1^m)
	R^n,
$$
we obtain
$$
	\int_{
		Q_{r_2}^{\tau - r_2^m, \tau - r_1^m}
	}
	u
	\,
	dx
	dt
	\le
	C
	(r_2^m - r_1^m)^{(p - 1) / p}
	R^{n (p - 1) / p}
	\left(
		\int_{
			Q_{r_2}^{\tau - r_2^m, \tau - r_1^m}
		}
		u^p
		\,
		dx
		dt
	\right)^{1 / p}.
$$
Thus,~\eqref{PL3.2.4} implies the estimate
\begin{equation}
	\int_{
		Q_{r_2}^{\tau - r_2^m, \tau - r_1^m}
	}
	u^p
	\,
	dx
	dt
	\ge
	\frac{
		C (r_2^m - r_1^m)
	}{
		R^{n (p - 1)}
	}
	\left(
		\int_{
			Q_{r_1}^{\tau - r_1^m, \tau}
		}
		f (x, t)
		g (u)
		\,
		dx
		dt
	\right)^p.
	\label{PL3.2.a4}
\end{equation}
Due to the obvious inequality 
$$
	r_2^m - r_1^m \ge C (r_2 - r_1) R^{m - 1}
$$
and relationship~\eqref{PL3.2.3} estimate~\eqref{PL3.2.a4} yields
\begin{equation}
	\int_{
		Q_{r_2}^{\tau - r_2^m, \tau - r_1^m}
	}
	u^p
	\,
	dx
	dt
	\ge
	C
	R^{m (p + 1) + n - 1}
	(r_2 - r_1)
	G^p (J_R (r_1, \tau))
	\essinf_{
		Q_{r_1}^{\tau - r_1^m, \tau}
	}
	f^p.
	\label{PL3.2.a3}
\end{equation}
Since 
$$
	Q_{r_2}^{\tau - r_2^m, \tau - r_1^m}
	\subset
	Q_{r_2}^{\tau - r_2^m, \tau}
	\setminus
	Q_{r_1}^{\tau - r_1^m, \tau},
$$
we have
\begin{align*}
	J_R (r_2, \tau) - J_R (r_1, \tau)
	&
	=
	\frac{
		1
	}{
		\mes Q_{2 R}^{\tau - 2^m R^m, \tau}
	}
	\int_{
		Q_{r_2}^{\tau - r_2^m, \tau}
		\setminus
		Q_{r_1}^{\tau - r_1^m, \tau}
	}
	u^p
	\,
	dx
	dt
	\\
	&
	\ge
	\frac{
		1
	}{
		\mes Q_{2 R}^{\tau - 2^m R^m, \tau}
	}
	\int_{
		Q_{r_2}^{\tau - r_2^m, \tau - r_1^m}
	}
	u^p
	\,
	dx
	dt.
\end{align*}
In view of~\eqref{PL3.2.a6}, this implies the estimate
$$
	J_R (r_2, \tau) - J_R (r_1, \tau)
	\ge
	\frac{
		C
	}{ 
		R^{m + n}
	}
	\int_{
		Q_{r_2}^{\tau - r_2^m, \tau - r_1^m}
	}
	u^p
	\,
	dx
	dt
$$
from which, taking into account~\eqref{PL3.2.a3}, we obtain~\eqref{L3.2.2}. 
\end{proof}

\begin{Lemma}\label{L3.3}
Let $R > 0$ and $\tau > 0$ be some real numbers such that $\tau > 2^m R^m$ and
\begin{equation}
	\int_{
		Q_R^{\tau - R^m, \tau}
	}
	u^p
	\,
	dx
	dt
	> 
	0.
	\label{L3.3.1}
\end{equation}
If $f (x, t) \ge 0$ for almost all $(x, t) \in Q_{2 R}^{\tau - 2^m R^m, \infty}$, then
\begin{align}
	&
	\int_{
		J_R (R, \tau)
	}^\infty
	G^{- 1 / m} (\zeta / 2) \zeta^{1 / m - 1}
	\,
	d\zeta
	+
	\int_{
		J_R (R, \tau)
	}^\infty
	\frac{
		d\zeta
	}{
		G (\zeta / 2)
	}
	+
	\int_{
		J_R (R, \tau)
	}^\infty
	\frac{
		d\zeta
	}{
		G^p (\zeta / 2)
	}
	\nonumber
	\\
	&
	\qquad
	\ge
	C
	\min 
	\left\{ 
		R
		\essinf_{
			Q_{2 R}^{\tau - 2^m R^m, \tau}
		}
		f^{1 / m},
		R^m
		\essinf_{
			Q_{2 R}^{\tau - 2^m R^m, \tau}
		}
		f,
		R^{m p}
		\essinf_{
			Q_{2 R}^{\tau - 2^m R^m, \tau}
		}
		f^p
	\right\}.
	\label{L3.3.2}
\end{align}
\end{Lemma}

\begin{proof}
We construct a finite sequence of real numbers $\{ r_i \}_{i=0}^l$ as follows.
Let us take $r_0 = R$. Assume further that $r_i$ is already defined.
If $r_i \ge 3 R / 2$, then we put $l = i$ and stop; otherwise we take
$$
	r_{i+1}
	=
	\sup
	\{
		r \in [r_i, 2 R] 
		:
		J_R (r, \tau)
		\le
		2 J_R (r_i, \tau)
	\}.
$$
Since $u \in L_{p, loc} ({\mathbb R}_+^{n+1})$, this procedure must terminate at a finite step. 
It follows from~\eqref{L3.3.1} that $J_R (r_i, \tau) > 0$ for all $0 \le i \le l$; 
therefore, $\{ r_i \}_{i=0}^l$ is a strictly increasing sequence.
It can also be seen that 
$$
	J_R (r_{i+1}, \tau) = 2 J_R (r_i, \tau)
	\quad
	\mbox{for all } 0 \le i \le l - 2
$$ 
and
$$
	J_R (r_l, \tau) \le 2 J_R (r_{l-1}, \tau).
$$
Moreover, if $J_R (r_l, \tau) < 2 J_R (r_{l-1}, \tau)$, then $r_l = 2 R$ and $r_l - r_{l-1} \ge R / 2$.

In view of Lemma~\ref{L3.2}, for any $0 \le i \le l - 1$ at least one of the following two
inequalities is valid:
\begin{equation}
	J_R (r_{i+1}, \tau) - J_R (r_i, \tau)
	\ge
	C 
	(r_{i+1} - r_i)^m
	G (J_R (r_i, \tau))
	\essinf_{
		Q_{2 R}^{\tau - 2^m R^m, \tau}
	}
	f,
	\label{PL3.3.1}
\end{equation}
\begin{equation}
	J_R (r_{i+1}, \tau) - J_R (r_i, \tau)
	\ge
	C 
	R^{m p - 1}
	(r_{i+1} - r_i)
	G^p (J_R (r_i, \tau))
	\essinf_{
		Q_{2 R}^{\tau - 2^m R^m, \tau}
	}
	f^p.
	\label{PL3.3.2}
\end{equation}

By $\Xi_1$ we denote the set of integers $0 \le i \le l - 1$ for which~\eqref{PL3.3.1} is valid.
In so doing, let $\Xi_2 = \{ 0, \ldots, l - 1 \} \setminus \Xi_1$.

We claim that
\begin{align}
	&
	\int_{
		J_R (r_i, \tau)
	}^{
		J_R (r_{i+1}, \tau)
	}
	G^{- 1 / m} (\zeta / 2) \zeta^{1 / m - 1}
	\,
	d\zeta
	+
	\int_{
		J_R (r_i, \tau)
	}^{
		J_R (r_{i+1}, \tau)
	}
	\frac{
		d\zeta
	}{
		G (\zeta / 2)
	}
	\nonumber
	\\
	&
	\qquad
	\ge
	C
	(r_{i+1} - r_i)
	\min 
	\left\{ 
		\essinf_{
			Q_{2 R}^{\tau - 2^m R^m, \tau}
		}
		f^{1 / m},
		R^{m - 1}
		\essinf_{
			Q_{2 R}^{\tau - 2^m R^m, \tau}
		}
		f
	\right\}
	\label{PL3.3.3}
\end{align}
for all $i \in \Xi_1$.
Really,~\eqref{PL3.3.1} implies the inequality
$$
	\left(
		\frac{
			J_R (r_{i+1}, \tau) - J_R (r_i, \tau)
		}{
			G (J_R (r_i, \tau))
		}
	\right)^{1 / m}
	\ge
	C
	(r_{i+1} - r_i)
	\essinf_{
		Q_{2 R}^{\tau - 2^m R^m, \tau}
	}
	f^{1 / m}.
$$
If $J_R (r_{i+1}, \tau) = 2 J_R (r_i, \tau)$, then due to monotonicity of the function $G$ we have
$$
	\int_{
		J_R (r_i, \tau)
	}^{
		J_R (r_{i+1}, \tau)
	}
	G^{- 1 / m} (\zeta / 2) \zeta^{1 / m - 1}
	\,
	d\zeta
	\ge
	C
	\left(
		\frac{
			J_R (r_{i+1}, \tau) - J_R (r_i, \tau)
		}{
			G (J_R (r_i, \tau))
		}
	\right)^{1 / m}.
$$
Combining the last two inequalities, we get
$$
	\int_{
		J_R (r_i, \tau)
	}^{
		J_R (r_{i+1}, \tau)
	}
	G^{- 1 / m} (\zeta / 2) \zeta^{1 / m - 1}
	\,
	d\zeta
	\ge
	C
	(r_{i+1} - r_i)
	\essinf_{
		Q_{2 R}^{\tau - 2^m R^m, \tau}
	}
	f^{1 / m},
$$
whence~\eqref{PL3.3.3} follows at once.

In turn, if $J_R (r_{i+1}, \tau) < 2 J_R (r_i, \tau)$ for some $i \in \Xi_1$, 
then $i = l - 1$ and, moreover, $r_{i+1} - r_i \ge R / 2$.
Thus,~\eqref{PL3.3.1} implies the estimate
$$
	\frac{
		J_R (r_{i+1}, \tau) - J_R (r_i, \tau)
	}{
		G (J_R (r_i, \tau))
	}
	\ge
	C
	R^{m - 1}
	(r_{i+1} - r_i)
	\essinf_{
		Q_{2 R}^{\tau - 2^m R^m, \tau}
	}
	f.
$$
Since
$$
	\int_{
		J_R (r_i, \tau)
	}^{
		J_R (r_{i+1}, \tau)
	}
	\frac{
		d\zeta
	}{
		G (\zeta / 2)
	}
	\ge
	\frac{
		C (J_R (r_{i+1}, \tau) - J_R (r_i, \tau))
	}{
		G (J_R (r_i, \tau))
	},
$$
this yields
$$
	\int_{
		J_R (r_i, \tau)
	}^{
		J_R (r_{i+1}, \tau)
	}
	\frac{
		d\zeta
	}{
		G (\zeta / 2)
	}
	\ge
	C
	R^{m - 1}
	(r_{i+1} - r_i)
	\essinf_{
		Q_{2 R}^{\tau - 2^m R^m, \tau}
	}
	f,
$$
whence we again derive~\eqref{PL3.3.3}.

In a similar way, it can be shown that
\begin{equation}
	\int_{
		J_R (r_i, \tau)
	}^{
		J_R (r_{i+1}, \tau)
	}
	\frac{
		d\zeta
	}{
		G^p (\zeta / 2)
	}
	\ge
	C
	R^{m p - 1}
	(r_{i+1} - r_i)
	\essinf_{
		Q_{2 R}^{\tau - 2^m R^m, \tau}
	}
	f^p
	\label{PL3.3.4}
\end{equation}
for all $i \in \Xi_2$. Really, taking into account~\eqref{PL3.3.2}, we have
$$
	\frac{
		J_R (r_{i+1}, \tau) - J_R (r_i, \tau)
	}{
		G^p (J_R (r_i, \tau))
	}
	\ge
	C
	R^{m p - 1}
	(r_{i+1} - r_i)
	\essinf_{
		Q_{2 R}^{\tau - 2^m R^m, \tau}
	}
	f^p.
$$
In view of the inequality
$$
	\int_{
		J_R (r_i, \tau)
	}^{
		J_R (r_{i+1}, \tau)
	}
	\frac{
		d\zeta
	}{
		G^p (\zeta / 2)
	}
	\ge
	\frac{
		C (J_R (r_{i+1}, \tau) - J_R (r_i, \tau))
	}{
		G^p (J_R (r_i, \tau))
	},
$$
this obviously implies~\eqref{PL3.3.4}.

Further, summing~\eqref{PL3.3.3} over all $i \in \Xi_1$, we obtain
\begin{align*}
	&
	\int_{
		J_R (r_0, \tau)
	}^{
		\infty
	}
	G^{- 1 / m} (\zeta / 2) \zeta^{1 / m - 1}
	\,
	d\zeta
	+
	\int_{
		J_R (r_0, \tau)
	}^{
		\infty
	}
	\frac{
		d\zeta
	}{
		G (\zeta / 2)
	}
	\\
	&
	\qquad
	\ge
	C
	\min 
	\left\{ 
		\essinf_{
			Q_{2 R}^{\tau - 2^m R^m, \tau}
		}
		f^{1 / m},
		R^{m - 1}
		\essinf_{
			Q_{2 R}^{\tau - 2^m R^m, \tau}
		}
		f
	\right\}
	\sum_{i \in \Xi_1}
	(r_{i+1} - r_i).
\end{align*}
Analogously,~\eqref{PL3.3.4} yields
$$
	\int_{
		J_R (r_0, \tau)
	}^{
		\infty
	}
	\frac{
		d\zeta
	}{
		G^p (\zeta / 2)
	}
	\ge
	C
	R^{m p - 1}
	\essinf_{
		Q_{2 R}^{\tau - 2^m R^m, \tau}
	}
	f^p
	\sum_{i \in \Xi_2}
	(r_{i+1} - r_i).
$$
Thus, summing the last two inequalities, we conclude that
\begin{align*}
	&
	\int_{
		J_R (r_0, \tau)
	}^\infty
	G^{- 1 / m} (\zeta / 2) \zeta^{1 / m - 1}
	\,
	d\zeta
	+
	\int_{
		J_R (r_0, \tau)
	}^\infty
	\frac{
		d\zeta
	}{
		G (\zeta / 2)
	}
	+
	\int_{
		J_R (r_0, \tau)
	}^\infty
	\frac{
		d\zeta
	}{
		G^p (\zeta / 2)
	}
	\\
	&
	\qquad
	\ge
	C
	\min 
	\left\{ 
		\essinf_{
			Q_{2 R}^{\tau - 2^m R^m, \tau}
		}
		f^{1 / m},
		R^{m - 1}
		\essinf_{
			Q_{2 R}^{\tau - 2^m R^m, \tau}
		}
		f,
		R^{m p - 1}
		\essinf_{
			Q_{2 R}^{\tau - 2^m R^m, \tau}
		}
		f^p
	\right\}
	(r_l - r_0).
\end{align*}
To complete the proof, it remains to note that $r_l - r_0 \ge R / 2$ 
by the construction of the sequence $\{ r_i \}_{i=0}^l$.
\end{proof}

We need the following known assertion.

\begin{Lemma}[see~{\cite[Lemma~2.3]{meIzv}}]\label{L3.4}
Let $\psi : (0,\infty) \to (0,\infty)$ and $\gamma : (0,\infty) \to (0,\infty)$
be measurable functions satisfying the condition
$$
	\gamma (\zeta)
	\le
	\essinf_{
		(\zeta / \theta, \theta \zeta)
	}
	\psi
$$
with some real number $\theta > 1$ for almost all $\zeta \in (0, \infty)$.
Also assume that 
$0 < \alpha \le 1$,
$M_1 > 0$,
$M_2 > 0$,
and
$\nu > 1$
are some real numbers with
$M_2 \ge \nu M_1$.
Then
$$
	\left(
		\int_{M_1}^{M_2}
		\gamma^{-\alpha} (\zeta)
		\zeta^{\alpha - 1}
		\,
		d \zeta
	\right)^{1 / \alpha}
	\ge
	A
	\int_{M_1}^{M_2}
	\frac{
		d \zeta
	}{
		\psi (\zeta)
	},
$$
where the constant $A > 0$ depends only on $\alpha$, $\nu$, and $\theta$.
\end{Lemma}

\begin{Lemma}\label{L3.5}
Let the hypotheses of Theorem~$\ref{T2.1}$ be valid, then
\begin{equation}
	\int_{
		Q_R^{\tau - R^m, \tau}
	}
	u^p
	\,
	dx
	dt
	\to
	0
	\quad
	\mbox{as } \tau \to \infty
	\label{L3.5.1}
\end{equation}
for any real number $R > 0$.
\end{Lemma}

\begin{proof}
It can easily be seen that~\eqref{T2.1.1} is equivalent to the condition
$$
	\int_1^\infty
	G^{- 1 / m} (\zeta / 2) \zeta^{1 / m - 1}
	\,
	d\zeta
	<
	\infty.
$$
By Lemma~\ref{L3.4}, we obtain
$$
	\left(
		\int_1^\infty
		G^{- 1 / m} (\zeta / 2) \zeta^{1 / m - 1}
		\,
		d\zeta
	\right)^m
	\ge
	C
	\int_1^\infty
	\frac{
		d\zeta
	}{
		G (\zeta)
	}.
$$
This allows us to assert that
$$
	\int_1^\infty
	\frac{
		d\zeta
	}{
		G (\zeta)
	}
	<
	\infty.
$$
Since $G$ is a non-decreasing positive function on the interval $(0, \infty)$, we have
$G^p (\zeta) \ge G^{p - 1} (1) G (\zeta)$ for all $\zeta \ge 1$.
Hence, we can also assert that
$$
	\int_1^\infty
	\frac{
		d\zeta
	}{
		G^p (\zeta)
	}
	<
	\infty.
$$

Assume that $R > 0$ is some given real number.
In view of~\eqref{T2.1.2}, $f$ is a non-negative function almost everywhere on 
$Q_{2 R}^{\tau - 2^m R^m, \infty}$ for all sufficiently large $\tau$. 
In so doing, it is obvious that the right-hand side of~\eqref{L3.3.2} tends to infinity 
as $\tau \to \infty$.
Thus, applying Lemma~\ref{L3.3}, we obtain 
$$
	J_R (R, \tau)
	\to
	0
	\quad
	\mbox{as } \tau \to \infty,
$$
whence~\eqref{L3.5.1} follows an once.
\end{proof}

\begin{proof}[Proof of Theorem~$\ref{T2.1}$]
Let $K$ be a compact subset of ${\mathbb R}^n$ and $R > 0$ be a real number 
such that $K \subset B_{R/2}$.
We denote
\begin{equation}
	U (t)
	=
	\int_{
		B_{R/2}
	}
	u (x, t)
	\,
	dx.
	\label{PT2.1.1}
\end{equation}
Since $u \in L_{1, loc} ({\mathbb R}_+^{n+1})$, the right-hand side of~\eqref{PT2.1.1} is defined 
for almost all $t \in (0, \infty)$ and, moreover, $U \in L_{1, loc} (0, \infty)$.
Let us put
$$
	U_h (\tau)
	=
	\int_0^\infty
	\omega_h (\tau - t)
	U (t)
	\,
	d t,
	\quad
	\tau > h > 0,
$$
where $\omega_h$ is given by~\eqref{3.1}. 
We have
$$
	\| U_h - U \|_{
		L_1 (H)
	}
	\to
	0
	\quad
	\mbox{as } h \to +0
$$
for any compact set $H \subset (0, \infty)$.
Hence, there exists a sequence of positive real numbers $\{ h_i \}_{i=1}^\infty$ such that
$h_i \to 0$ as $i \to \infty$ and
$$
	\lim_{i \to \infty} 
	U_{h_i} (\tau) 
	=
	U (\tau)
$$
for almost all $\tau \in (0, \infty)$.
Also let $\tau_0 > R^m$ be a real number such that $f$ is a non-negative function
almost everywhere on $Q_R^{\tau_0 - R^m, \infty}$.
In view of~\eqref{T2.1.2}, such a real number obviously exists. 

Lemma~\ref{L3.1} with $r_1 = R / 2$, $r_2 = R$, $\tau_1 = (R / 2)^m$, and $\tau_2 = R^m$ yields
\begin{align}
	&
	\frac{1}{R^m}
	\int_{
		Q_R^{\tau - R^m, \tau + h_i}
	}
	u^p
	\,
	dx
	dt
	+
	\frac{1}{R^m - (R / 2)^m}
	\int_{
		Q_R^{\tau - R^m, \tau - (R / 2)^m}
	}
	u
	\,
	dx
	dt
	\nonumber
	\\
	&
	\qquad
	\ge
	C
	\int_{
		Q_{R / 2}^{\tau - h_i, \tau + h_i}
	}
	\omega_{h_i} (\tau - t)
	u 
	\,
	dx
	dt
	\label{PT2.1.3}
\end{align}
for all $\tau > \tau_0$ and for all $i$ such that $h_i < (R / 2)^m$.
Note that the first summand in the right-hand side of~\eqref{L3.1.1} is non-negative; 
therefore, it can be dropped. 

Since
$$
	\int_{
		Q_{R / 2}^{\tau - h_i, \tau + h_i}
	}
	\omega_{h_i} (\tau - t)
	u 
	\,
	dx
	dt
	=
	U_{h_i} (\tau),
	\quad
	\tau > h_i,
$$
passing to the limit as $i \to \infty$ in~\eqref{PT2.1.3}, we obtain
\begin{equation}
	\frac{1}{R^m}
	\int_{
		Q_R^{\tau - R^m, \tau}
	}
	u^p
	\,
	dx
	dt
	+
	\frac{1}{R^m - (R / 2)^m}
	\int_{
		Q_R^{\tau - R^m, \tau - (R / 2)^m}
	}
	u
	\,
	dx
	dt
	\ge
	C
	U (\tau)
	\label{PT2.1.4}
\end{equation}
for almost all $\tau \in (\tau_0, \infty)$.
By the H\"older inequality,
$$
	\int_{
		Q_R^{\tau - R^m, \tau - (R / 2)^m}
	}
	u
	\,
	dx
	dt
	\le
	\int_{
		Q_R^{\tau - R^m, \tau}
	}
	u
	\,
	dx
	dt
	\le
	C
	R^{(m + n) (p - 1) / p}
	\left(
		\int_{
			Q_R^{\tau - R^m, \tau}
		}
		u^p
		\,
		dx
		dt
	\right)^{1 / p}.
$$
Thus,~\eqref{PT2.1.4} implies the estimate
$$
	R^{-m}
	\int_{
		Q_R^{\tau - R^m, \tau}
	}
	u^p
	\,
	dx
	dt
	+
	R^{n - (m + n) / p}
	\left(
		\int_{
			Q_R^{\tau - R^m, \tau}
		}
		u^p
		\,
		dx
		dt
	\right)^{1 / p}
	\ge
	C
	U (\tau)
$$
for almost all $\tau \in (\tau_0, \infty)$.
To complete the proof, it remains to use Lemma~\ref{L3.5}.
\end{proof}

\end{document}